\documentclass{article}

\usepackage{amsmath, amsthm}
\usepackage{url} 

\usepackage{tikz}

\newcommand{\QQ}{\mathbf Q}
\newcommand{\ZZ}{\mathbf Z}

\newcommand{\PP}{\mathbf P}
\newcommand{\NN}{\mathbf N}
\newcommand{\MM}{\mathbf M}
\newcommand{\AAA}{\mathbf A}
\newcommand{\BB}{\mathbf B}
\newcommand{\CC}{\mathbf C}
\newcommand{\XX}{\mathbf X}
\newcommand{\YY}{\mathbf Y}

\newcommand{\vv}{\mathbf v}
\newcommand{\ww}{\mathbf w}

\newcommand{\ee}{\mathbf e}
\newcommand{\pp}{\mathbf p}
\newcommand{\nn}{\mathbf n}
\newcommand{\II}{\mathbf I}

\newcommand{\qq}{\mathbf q}
\newcommand{\rr}{\mathbf r}
\newcommand{\ttt}{\mathbf t}

\newcommand{\zzero}{\mathbf 0}

\newtheorem{theorem}{Theorem}[section]
\newtheorem{lemma}[theorem]{Lemma} 
\newtheorem{corollary}[theorem]{Corollary}

\newtheorem{proposition}[theorem]{Proposition}

\theoremstyle{definition}

\newtheorem{example}[theorem]{Example}

\theoremstyle{remark}

\numberwithin{equation}{section}

\title{Zariski decomposition: a new (old) chapter of linear algebra}

\author{Thomas Bauer, Mirel Caib\u{a}r, and Gary Kennedy}

\date{}
\begin{document}
\maketitle


\begin{abstract}
In a 1962 paper, Zariski introduced the decomposition theory that now bears his name. Although it arose in the context of algebraic geometry and deals with the configuration of curves on an algebraic surface, we have recently observed that the essential concept is purely within the realm of linear algebra. In this paper, we formulate Zariski decomposition as a theorem in linear algebra and present a linear algebraic proof. We also sketch the geometric context in which Zariski first introduced his decomposition.
\end{abstract}

\section{Introduction.}
Oscar Zariski (1899--1986) was a central figure in 20th century mathematics. His life, ably recounted in \cite{parikh}, took him from a small city in White Russia, through his advanced training under the masters of the ``Italian school'' of algebraic geometry, and to a distinguished career in the United States, the precursor of a tide of emigrant talent fleeing political upheaval in Europe. As a professor at Johns Hopkins and Harvard University, he supervised the Ph.D.s of some of the most outstanding mathematicians of the era, including two Fields Medalists, and his mathematical tribe (traced through advisors in \cite{gen}) now numbers more than 800. 
Zariski thoroughly absorbed and built upon the synthetic arguments of the Italian school, and in \cite{algsurf} he gave a definitive account of the classical theory of algebraic surfaces. In the course of writing this volume, however, despite his admiration for their deep geometric insight he became increasingly disgruntled with the lack of rigor in certain arguments. He was thus led to search for more adequate foundations for algebraic geometry, taking (along with Andre Weil) many of the first steps in an eventual revolutionary recasting of these foundations by Alexander Grothendieck and others.
\par
In a 1962 paper \cite{zar}, Zariski introduced the decomposition theory that now bears his name. Although it arose in the context of algebraic geometry and deals with the configuration of curves on an algebraic surface, we have recently observed that the essential concept is purely within the realm of linear algebra. (A similar observation has been made independently by Moriwaki in Section 1 of \cite{moriwaki}.)
In this paper, we formulate Zariski decomposition as a theorem in linear algebra and present a linear algebraic proof.  To motivate the construction, however, we begin in Section \ref{origcontext} with a breezy account of the the original geometric situation, and eventually return to this situation in Section \ref{windup} to round off the discussion and present one substantive example. We give only a sketchy description which lacks even proper definitions; one needs a serious course in algebraic geometry to treat these matters in a rigorous way. But, as already indicated, the thrust of the paper is in a far different direction, namely toward disentangling the relatively elementary linear algebra from these more advanced ideas. Beginning in Section \ref{decomp}, our treatment is both elementary and explicit; a basic course in linear algebra, which includes the idea of a negative definite matrix, should be a sufficient background.
After laying out the definitions and the main idea, we present a simple new construction (which first appeared in \cite{simpleproof}) and show that it satisfies the requirements for a Zariski decomposition. We look at a few elaborations, and we present Zariski's original algorithm (shorn of its original geometric context).


\section{The original context.} \label{origcontext}

The study of algebraic curves, with its ties to the theory of Riemann surfaces and many other central ideas of mathematics, has ancient roots, but our understanding of algebraic surfaces has developed more recently. One of Zariski's main concerns was how to extend well-known fundamental theories from curves to surfaces. In trying to understand such a surface, one is naturally led to study the algebraic curves which live on it, asking what sorts of curves there are, how they meet each other, and how their configurations influence the geometry of the surface.  For example, in the plane\footnote{We mean the complex projective plane. Our equation is given in affine coordinates, but we intend for the curve to include appropriate points at infinity. The reader who hasn't encountered these notions will need to take our assertions in this section on faith.}
(the simplest example of an algebraic surface) an algebraic curve is the solution set of a polynomial equation $f(x,y)=0$. One can calculate that the vector space of all polynomials in two variables of degree not exceeding $d$ is a vector space of dimension $\binom{d+2}{2}$. Since two such polynomials define the same curve if and only if one is a multiple of the other, we say that the set of all such curves forms a \emph{linear system} of dimension $\binom{d+2}{2}-1$. (In general, the dimension of a linear system is one less than the dimension of the corresponding vector space of functions.) More generally, for each curve $D$ on an algebraic surface one can naturally define an associated linear system of curves 
which are equivalent in a certain sense to $D$, denoting it by $|D|$. This linear system depends not just on the curve as a set of points but also on the equation which defines it: the equation $f(x,y)^n=0$ defines a larger linear system than does $f(x,y)=0$, and we denote this larger linear system by $|nD|$. (For a curve of degree $d$ in the plane, $|nD|$ consists of all curves of degree $nd$.)
\par
His student David Mumford (in an appendix to \cite{parikh}) says that 
``Zariski's papers on the general topic of linear systems form a rather coherent whole in which one can observe at least two major themes which he developed repeatedly. One is the Riemann-Roch problem: to compute the dimension of a general linear system \dots and especially to consider the behavior of $\dim |nD|$ as $n$ grows. The other is to apply the theory of linear systems in the 2-dimensional case to obtain results on the birational geometry of surfaces and on the classification of surfaces. In relation to his previous work, this research was, I believe, something like a dessert. He had worked long setting up many new algebraic techniques and laying rigorous foundations for doing geometry --- and linear systems, which are the heart of Italian geometry, could now be attacked."
\par
Zariski's paper \cite{zar} is concerned with the following question: for a specified curve $D$ on an algebraic surface, what is the order of growth of $\dim |nD|$ as a function of $n$? His answer involved a decomposition: he showed that $D$, considered as an element of a certain vector space, could be written as a sum $P+N$ of a ``positive part'' and a ``negative part," so that the answer to his question was determined by $P$ alone. Specifically, he showed that the order of growth was the ``self-intersection number'' of $P$. In the heart of this paper, we will give an account of Zariski's decomposition, assuming that we already have been given the relevant ``intersection theory'' on the surface. In the last section of the paper we will resume this account of the original context. In particular we will say something about how this intersection theory arises, and give a precise statement of Zariski's formula on the order of growth.
\begin{figure} 
\begin{center}
\includegraphics[width=60mm]{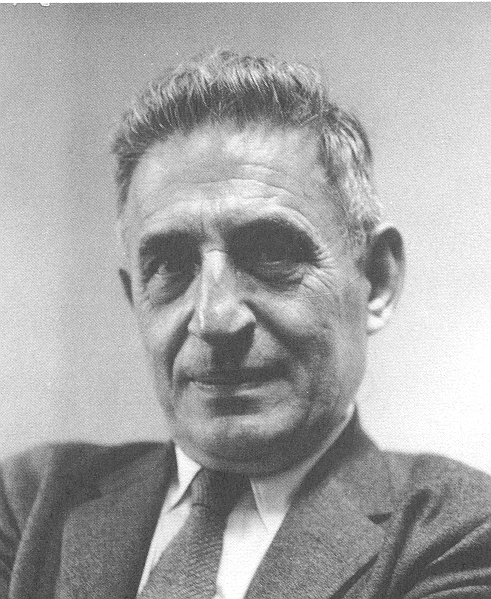}
\caption{Oscar Zariski in 1960 (frontispiece photo of \cite{parikh}, credited to Yole Zariski)}.
\end{center}
\end{figure}


\section{The decomposition.} \label{decomp}
We now forget about the original context, and lay out an elementary theory within linear algebra. In Section \ref{windup} we will resume our account of the geometry which motivates the following definitions.
\par
Suppose that $V$ is a vector space over $\QQ$ (the rational numbers) equipped with a symmetric bilinear form; we denote the product of $\vv$ and $\ww$ by $\vv \cdot \ww$. Suppose furthermore that there is a basis $E$ with respect to which the 
bilinear form is an \emph{intersection product}, meaning that the product of any two distinct basis elements is nonnegative.
Most of our examples will be finite-dimensional, but we are also interested in the infinite-dimensional case.
If $V$ is finite-dimensional then we will assume that $E$ has an ordering, and using this ordered basis we will identify $V$ with $\QQ^n$ (where $n$ is its dimension); a vector $\vv$ will be identified with its coordinate vector, written as a column. We can then specify the bilinear form by writing its associated symmetric matrix $\MM$ with respect to the basis, calling it the \emph{intersection matrix}. Thus the product of $\vv$ and $\ww$ is $\vv^T \MM \ww$, where $T$ denotes the transpose. With this interpretation, the form is an intersection product if and only if all off-diagonal entries of $\MM$ are nonnegative. 
\par
In any case, whether $V$ is finite- or infinite-dimensional, each element $\vv \in V$ can be written in a unique way as a linear combination of a finite subset of the basis, with all coefficients nonzero.
We will call this finite subset the \emph{support} of $\vv$, and the finite-dimensional subspace of $V$ which it spans is called the \emph{support space} of $\vv$. If all coefficients are positive, then 
$\vv $ is said to be \emph{effective}.\footnote{In the motivating application, the basis vectors
will be certain curves on the algebraic surface, and hence an arbitrary vector
$\vv \in V$ will be a linear combination of such curves. The combinations
that use nonnegative coefficients may be interpreted geometrically,
while the others are just ``virtual curves.''}
In particular each basis element is effective, and the zero vector is also considered to be effective (since we may sum over the empty set).
\par
A vector $\ww$ is called \emph{nef with respect to} $V$ if $\ww \cdot \vv \geq 0$ for every effective vector $\vv$. Note that to check whether a vector satisfies this condition it suffices to check whether its product with each basis element is nonnegative. In the finite-dimensional case (using the identification of $V$ with $\QQ^n$, as described above) the definition can be formulated in terms of the intersection matrix: since the entries of $\MM\ww$ are the products of the basis elements with $\ww$, we observe that a vector $\ww$ is nef with respect to $V$ precisely when $\MM\ww$ is effective.
In particular if $\MM$ is nonsingular, then $\ww$ is nef with respect to $V$ if and only if there is an effective vector $\vv \in V$ for which $\MM^{-1}\vv=\ww$.
\par
Now suppose that $W$ is a subspace of $V$ 
spanned by some subset of the basis and containing the support space of a vector $\ww$ (for example, $W$ could be the support space itself). If $\ww$ is nef with respect to $V$ then it is nef with respect to $W$, but the opposite implication may not be correct.
\begin{example} \label{nefnotnef}
Suppose that the intersection matrix is
$$
\MM =
\left[
\begin{array}{rr}
-2 & 1 \\
1 & 1 
\end{array}
\right],
$$
and let $W$ be the one-dimensional subspace spanned by the first basis element $\ee_1$. Then $-\ee_1$ is nef with respect to $W$, but it is not nef with respect to $V$.
\end{example}
\par
We do, however, have  a partial converse.
\begin{lemma} 
If $\ww \in W$ is effective and nef with respect to the subspace $W$, 
then it is nef with respect to the entire space $V$.
\end{lemma}
\begin{proof}
By hypothesis, the product of $\ww$ and a basis element for $W$ is nonnegative. Since $\ww$ is effective, its intersection product with any other basis element of $V$ is likewise nonnegative.
\end{proof}
In view of this lemma, we may simply call such a vector  \emph{effective and nef}.\footnote{Some say that the neologism ``nef'' is short for ``numerically effective,'' but this gives a misleading impression of its meaning (since an effective vector is not necessarily nef). Others insist that it should be thought of as an acronym for ``numerically eventually free.''}
\par
Here is our main theorem.
\begin{theorem} \label{zdecomp}
For each effective element $\vv \in V$, there is a unique way to write it as a sum 
$$
\vv=\pp+\nn
$$
of elements satisfying the following conditions:
\begin{enumerate}
\item $\pp$ is nef with respect to $V$;
\item $\nn$ is effective;
\item $\pp \cdot \ee=0$ for each basis element $\ee$ in the support of $\nn$;
\item the restriction of the intersection product to the support space of $\nn$ is negative definite.
\end{enumerate}
Furthermore $\pp$ is effective.
\end{theorem}
This is called the \emph{Zariski decomposition} of $\vv$; the elements $\pp$ and $\nn$ are called its \emph{positive} and \emph{negative} parts. We note that both extremes are possible: for example, if $\vv$ itself is nef with respect to $V$,
 then $\pp=\vv$ and the support space of $\nn$ is trivial.
\par
\begin{example} \label{myexample}
Again suppose that
$$
\MM =
\left[
\begin{array}{rr}
-2 & 1 \\
1 & 1 \\
\end{array}
\right],
$$
and let $\vv=2\ee_1+\ee_2$. Since $\vv \cdot \vv=-3$, the vector $\vv$ is not nef. But since $\ee_2\cdot\ee_2$ is positive, $\ee_2$ cannot be in the support of $\nn$. Thus $\nn=x\ee_1$ and $\pp=(2-x)\ee_1+\ee_2$ for some number $x$. By the third condition $\pp\cdot\ee_1=-2(2-x)+1=0$. Thus
$$
\pp=\frac{1}{2}\ee_1+\ee_2
\quad
\text{and}
\quad
\nn=\frac{3}{2}\ee_1.
$$
It's instructive to look at all elements $x\ee_1+y\ee_2$, where $x \leq 2$ and $y \leq 1$. (Since the coordinates of $\nn$ must be nonnegative, these are the only possibilities for $\pp$.) If the corresponding points $(x,y)$ are plotted in the plane, then the nef elements form a triangle, and the element $\pp$ corresponds to the upper right vertex. See Figure~\ref{examplecone}.

\begin{figure}
\begin{center}
\begin{tikzpicture} 
\filldraw [lightgray] (-1.5,-.5) rectangle (2,1);
\draw [->] (0,0) -- (2.5,0) node[right]{$x$}; 
\draw [->] (0,0) -- (0,1.7) node[above]{$y$}; 
\draw (2,1) node[above right]{$\vv$} -- (2,-0.5);
\draw (2,1) -- (-1.5,1);
\filldraw (2,1) circle (2pt);
\filldraw [lightgray] (5,0) -- (5.75,1.5) -- (3.5,1.5) -- cycle;
\draw [->] (5,0) -- (7.5,0) node[right]{$x$}; 
\draw [->] (5,0) -- (5,1.7) node[above]{$y$};
\draw (4.5,-1) -- (5.75,1.5);
\draw (3.5,1.5) -- (6,-1);
\filldraw [lightgray] (10,0) -- (10.5,1) -- (9,1) -- cycle;
\draw [->] (10,0) -- (12.5,0) node[right]{$x$}; 
\draw [->] (10,0) -- (10,1.7) node[above]{$y$};
\draw (12,1) node[above right]{$\vv$} -- (12,-0.5);
\draw (12,1) -- (8.5,1);
\draw (9.5,-1) -- (10.75,1.5);
\draw (8.5,1.5) -- (11,-1);
\filldraw (12,1) circle (2pt);
\filldraw (10.5,1) circle (2pt) node[below right]{$\pp$};
\filldraw (11.5,0) circle (2pt) node[above left]{$\nn$};
\end{tikzpicture}
\caption{An example of Zariski decomposition. The picture on the left shows the candidates for the positive part of $\vv$. The middle picture shows the nef vectors. The shaded triangle in the right picture is their overlap.}
\label{examplecone}
\end{center}
\end{figure}
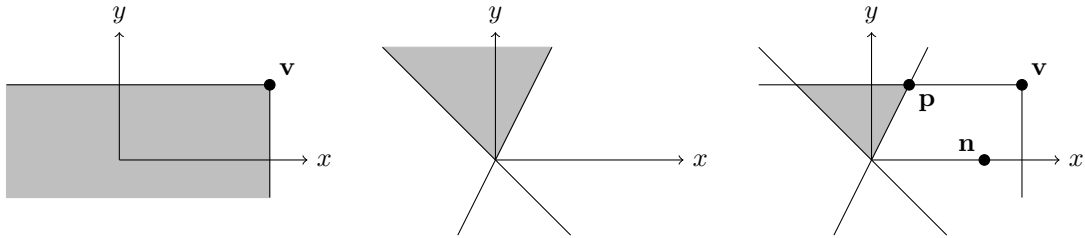
\end{example}

\section{Proof of the main theorem.}

Recall that if $V$ is finite-dimensional then we will identify it with $\QQ^n$. In particular the basis element $\ee_j$ is identified with the column vector having 1 in position $j$ and 0 elsewhere.
We begin the proof with a pair of lemmas.
\begin{lemma} \label{negdefinverse}
If $\MM$ is a negative definite matrix whose off-diagonal entries are nonnegative, then all entries of $\MM^{-1}$ are nonpositive.
\end{lemma}
\begin{proof}
(adapted from the Appendix of \cite{chambers})
Write $\MM^{-1}\ee_j$ as a difference of effective vectors $\qq-\rr$ with no common support vector.
Then $\qq^T \MM \rr \geq 0$. Hence (since $\MM$ is negative definite) for $\qq \neq \zzero$ we have
$$
\qq^T \MM \qq-\qq^T \MM \rr < 0.
$$
But this is $\qq^T \ee_j$, the $j$th entry of $\qq$, which is nonnegative.
Thus $\qq=\zzero$, which says that all the entries of column $j$ of $\MM^{-1}$ are nonpositive.
\end{proof}
\begin{lemma} \label{existlemma}
Suppose $\MM$ is a symmetric matrix whose off-diagonal entries are nonnegative. Suppose that $\MM$ is not negative definite. Then there is a nonzero vector $\qq$ for which $\qq$ and $\MM\qq$ are both effective.
\end{lemma}
\begin{proof}
If the top left entry of $\MM$ is nonnegative then we can take $\qq=\ee_{1}$. Otherwise let $\MM'$
be the largest upper left square submatrix which is negative definite, and write
$$
\MM =
\left[
\begin{array}{cc}
\MM' & \AAA \\
\AAA^{T} & \BB \\
\end{array}
\right].
$$
Denote the dimension of $\MM'$ by $m'$.
Since $\MM'$ is nonsingular, there is a vector 
$$
\qq=\left[
\begin{array}{c}
\qq' \\
1 \\
0 \\
\vdots \\
0 \\
\end{array}
\right]
$$
in the kernel of the map defined by
$\left[
\begin{array}{cc}
\MM' & \AAA \\
\end{array}
\right]$,
where $\qq'$ has length $m'$.
Letting $\AAA_1$ denote the first column of $\AAA$, we see that
$
\MM' \qq'=-\AAA_1,
$
and thus $\qq'=-{\MM'}^{-1}\AAA_1$.
By Lemma \ref{negdefinverse} we see that all entries of $\qq'$ are nonnegative. 
Thus the same is true of $\qq$.
\par
Turning to $\MM \qq$, we know that it begins with $m'$ zeros. Thus the product 
$\qq^T \MM \qq$ computes entry $m'+1$ of $\MM \qq$.
Now note that by the choice of $\MM'$ there is a vector 
$$
\ww=\left[
\begin{array}{c}
\ww' \\
1 \\
0 \\
\vdots \\
0 \\
\end{array}
\right]
$$
(with $\ww'$ of length $m'$) for which $\ww^T \MM \ww \geq 0$. An easy calculation shows that $(\qq-\ww)^T \MM \qq = 0$, and by transposition we have $\qq^T \MM (\qq-\ww) = 0$. Also note that $(\qq-\ww)^T \MM (\qq-\ww) \leq 0$, since 
$\qq-\ww$ belongs to a subspace on which the associated bilinear form is negative definite.
Thus by bilinearity
$$
\qq^T \MM \qq = 
(\qq-\ww)^T \MM \qq
- (\qq-\ww)^T \MM (\qq-\ww)
+ \ww^T \MM \ww
+ \qq^T \MM (\qq-\ww)
\geq 0.
$$
As for the remaining entries of $\MM \qq$, each one is a sum of products of nonnegative numbers; thus these entries are all nonnegative.
\end{proof}
\begin{corollary} \label{restate}
Suppose that the restriction of an intersection product to a finite-dimensional subspace is not negative definite. Then there is a nonzero effective and nef element in this subspace.
\end{corollary}
\par
We now present a procedure for constructing the Zariski decomposition of an effective element
$\vv=\sum_{i=1}^{n}c_{i}\ee_{i}$.
We will momentarily allow arbitrary real numbers as coefficients, but 
we will soon show that rational coefficients suffice.
Consider a ``candidate'' for the positive part: $\sum_{i=1}^{n}x_{i}\ee_{i}$,
where 
\begin{equation} \label{box}
x_i \leq c_i 
\end{equation} 
for each $i$.
(Look back at Figure \ref{examplecone} for motivation.)
Such an element is nef if and only if the inequality
\begin{equation} \label{cone}
\sum_{i=1}^{n} x_i (\ee_i \cdot \ee_j) \geq 0
\end{equation}
is satisfied for each $j$. 
Consider the set defined by the $2n$ inequalities in (\ref{box}) and (\ref{cone}),
together with the $n$ additional conditions $x_i \geq 0$.  Since this set is compact and nonempty (it contains the zero vector),
there is at least one point where $\sum_{i=1}^{n} x_i$ is maximized.
Let $\pp$ be the corresponding element of $V$, and let $\nn=\vv-\pp$.
We claim that this is a Zariski decomposition.
\par
By construction, the first two conditions in Theorem  \ref{zdecomp} are satisfied. Regarding the third condition, note (since $\pp$ maximizes $\sum x_i$)
that if $\ee_j$ is in the support of $\nn$ then, 
for  $\epsilon>0$ and sufficiently small, the element $\pp+\epsilon \ee_j$ is not nef. But  $(\pp+\epsilon \ee_j)\cdot \ee_i \geq 0$ for all $i \neq j$. 
Thus $(\pp+\epsilon \ee_j)\cdot \ee_j < 0$ for all positive $\epsilon$, and this implies that $\pp \cdot \ee_j \leq 0$. Since $\pp$ is nef we have $\pp \cdot \ee_j = 0$.
\par
To prove that the restriction of the intersection product to the support space of $\nn$ is negative definite, we argue by contradiction. Supposing that the restriction of the form is not negative definite, Corollary \ref{restate} tells us that there is a nonzero effective and nef element $\qq$ in the support space of $\nn$. 
Then for small $\epsilon>0$ the element $\pp+\epsilon\qq$ is nef and $\nn-\epsilon\qq$ is effective. But this contradicts the maximality of $\pp$.
\par
To prove the remaining claims of Theorem  \ref{zdecomp} (and the implicit claim that all coefficients of $\pp$ and $\nn$ are rational numbers), we need the following idea.
Define the \emph{maximum} of two elements $\vv=\sum_{i=1}^{n}x_{i}\ee_{i}$ and $\vv'=\sum_{i=1}^{n}x'_{i}\ee_{i}$ by $\max(\vv,\vv')=\sum_{i=1}^{n}\max(x_i,x'_{i})\ee_{i}$.
\begin{lemma} \label{maxnef}
If $\pp$ and $\pp'$ are both nef, then so is $\max(\pp,\pp')$.
\end{lemma}
\begin{proof}
The $j$th inequality in (\ref{cone}) involves at most one negative coefficient, namely $\ee_j \cdot \ee_j$. Suppose that  $\pp=\sum_{i=1}^{n}x_{i}\ee_{i}$ and $\pp'=\sum_{i=1}^{n}x'_{i}\ee_{i}$ satisfy this inequality. We may assume that $x_j \geq x'_j$. Then $\max(\pp,\pp')-\pp$ satisfies the inequality; hence $\max(\pp,\pp')$ satisfies it as well.
\end{proof}
Here is the proof of uniqueness. 
Suppose that $\vv=\pp+\nn$ and $\vv=\pp'+\nn'$ are two Zariski decompositions of $\vv$. 
Let $\max(\pp,\pp')=\pp+\sum x_{i}\ee_{i}$, 
where the sum is over the support of $\nn$ and the coefficients are nonnegative. Since $\max(\pp,\pp')$ is nef, we know that for each element $\ee_j$ of the support of $\nn$ we have
$$
\sum x_{i}\ee_{i} \cdot \ee_{j} = \max(\pp,\pp') \cdot \ee_{j} \geq 0.
$$
Thus
$$
\sum x_{i}\ee_{i} \cdot \sum x_{j}\ee_{j}=\sum \sum x_{i}x_{j}\ee_{i} \cdot \ee_{j} \geq 0.
$$
Since the intersection product is negative definite on the support space of $\nn$, all $x_i=0$. Thus $\pp=\max(\pp,\pp')$. Similar reasoning shows that $\pp'=\max(\pp,\pp')$, and thus $\pp=\pp'$.
\par
Having uniqueness, we can now note that by our construction the positive part of the Zariski decomposition is an effective vector.
\par
Finally we argue that the positive and negative parts have rational coefficients.
Let $\pp=\sum_{i=1}^{n}x_{i}\ee_{i}$.
Then its coefficients satisfy $n$ linear equations, namely:
\begin{gather*}
\sum_{i=1}^{n} x_i (\ee_i \cdot \ee_j) = 0
\text{ for each basis element $\ee_j$ in the support of $\nn$}, \\
x_j = c_j
\text{ for each basis element $\ee_j$ not in the support of $\nn$}.
\end{gather*}
In matrix form (and with the basis suitably reordered), we have the following equation:
$$
\left[
\begin{array}{cc}
\NN & \AAA \\
\zzero & \II \\
\end{array}
\right]\XX=\left[
\begin{array}{c}
\zzero \\
\CC \\
\end{array}
\right],
$$
where $\NN$ is negative definite, $\zzero$ is a zero matrix, and $\II$ is an identity matrix. This is a nonsingular system in which all entries are rational numbers, 
and we know that its unique solution gives the positive part of the Zariski decomposition.


\section{Zariski's original algorithm.}
Our construction gives the Zariski decomposition of an effective vector in one fell swoop. In Zariski's original paper, by contrast, he built up the negative part in stages.\footnote{This comparison is somewhat unfair, since our construction simply says to maximize a certain linear function on a polytope. To actually discover the location of the maximum one would have to invoke a step-by-step algorithm such as the simplex method.}
Our exposition of his algorithm relies on the last chapter of \cite{bad}. Let us call a finite subspace of $V$ a \emph{special subspace} if it is spanned by a subset of the basis. 
We say that a subspace is \emph{negative definite} if the restriction of the intersection product to the subspace is negative definite.
The basic idea is to work toward the correct support space for the negative part of the specified vector, through an increasing sequence of negative definite special subspaces. 
\begin{example} \label{algoex}
Suppose that $V$ is finite-dimensional with intersection matrix
$$
\MM =
\left[
\begin{array}{rrrr}
-2 & 0 & 1 & 1 \\
0 & -2 & 1 & 2 \\\
1 & 1 & -2 & 0 \\\
1 & 2 & 0 & -2
\end{array}
\right].
$$
Figure \ref{lattice} shows the lattice of negative definite subspaces. In Example~\ref{algexample} we will show how Zariski's algorithm hunts through this lattice.
\end{example}
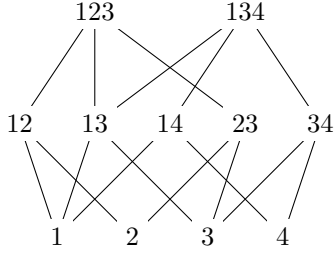
\begin{figure}
\begin{center}
\begin{tikzpicture} 
\draw (0.5,0) node [fill=white] {1} -- (0,1.5) node [fill=white] {12};
\draw (0.5,0) node [fill=white] {1} -- (1,1.5) node [fill=white] {13};
\draw (0.5,0) node [fill=white] {1} -- (2,1.5) node [fill=white] {14};
\draw (1.5,0) node [fill=white] {2} -- (0,1.5) node [fill=white] {12};
\draw (1.5,0) node [fill=white] {2} -- (3,1.5) node [fill=white] {23};
\draw (2.5,0) node [fill=white] {3} -- (1,1.5) node [fill=white] {13};
\draw (2.5,0) node [fill=white] {3} -- (3,1.5) node [fill=white] {23};
\draw (2.5,0) node [fill=white] {3} -- (4,1.5) node [fill=white] {34};
\draw (3.5,0) node [fill=white] {4} -- (2,1.5) node [fill=white] {14};
\draw (3.5,0) node [fill=white] {4} -- (4,1.5) node [fill=white] {34};
\draw (0,1.5) node [fill=white] {12} -- (1,3) node [fill=white] {123};
\draw (1,1.5) node [fill=white] {13} -- (1,3) node [fill=white] {123};
\draw (3,1.5) node [fill=white] {23} -- (1,3) node [fill=white] {123};
\draw (1,1.5) node [fill=white] {13} -- (3,3) node [fill=white] {134};
\draw (2,1.5) node [fill=white] {14} -- (3,3) node [fill=white] {134};
\draw (4,1.5) node [fill=white] {34} -- (3,3) node [fill=white] {134};
\end{tikzpicture}
\caption{The lattice of negative definite subspaces in Example \ref{algoex}. The subspace spanned by basis vectors $\ee_1$ and $\ee_3$, for example, is indicated by 13.}
\label{lattice}
\end{center}
\end{figure}
The algorithm relies on three lemmas.
\begin{lemma} \label{firstlem}
(cf. Lemma 14.9 of \cite{bad}) 
Let $N$ be a negative definite special subspace, and suppose that $\nn \in N$ is a vector for which $\nn \cdot \ee \leq 0$ for every basis element $\ee \in N$. Then $\nn$ is effective.
\end{lemma}
\begin{proof}
As in the proof of Lemma \ref{negdefinverse},
write $\nn=\qq-\rr$, where $\qq$ and $\rr$ are effective but
have no common support vector.
Then $\qq \cdot \rr \geq 0$. Hence
$$
\rr \cdot \rr \geq \rr \cdot \rr - \qq \cdot \rr = -\nn \cdot \rr \geq 0.
$$
Since the subspace is negative definite this implies that $\rr=\zzero$, i.e., that $\nn$ is effective.
\end{proof}

\begin{lemma} \label{secondlem}
Suppose that $\vv \in V$ is an effective vector. Suppose that $N$ is a negative definite special subspace of the support space of $\vv$. Suppose that $\pp$ is a vector satisfying these two conditions:
\begin{enumerate}
\item $\pp \cdot \ee=0$ for each basis element $\ee \in N$;
\item $\vv-\pp$ is an element of $N$.
\end{enumerate}
Then $\pp$ is effective.
\end{lemma}
\begin{proof}
Work in the support space of $\vv$, which is finite-dimensional. Rearrange the basis so that the intersection matrix is 
$$
\left[
\begin{array}{cc}
\MM & \AAA \\
\AAA^{T} & \BB \\
\end{array}
\right],
$$
where $\MM$ is the negative definite intersection matrix for the subspace $N$.
Write $\pp$ as a column matrix with respect to this basis:
$$
\pp=\left[
\begin{array}{c}
\XX \\
\YY \\
\end{array}
\right].
$$
Then, since $\pp \cdot \ee=0$ for each basis element $\ee \in N$,
$$
\left[
\begin{array}{cc}
\MM & \AAA \\
\AAA^{T} & \BB \\
\end{array}
\right]
\left[
\begin{array}{c}
\XX \\
\YY \\
\end{array}
\right]
=
\left[
\begin{array}{c}
\zzero \\
\ZZ \\
\end{array}
\right],
$$
and thus
$\XX=-{\MM}^{-1}\AAA\YY$.
We know that all entries of $\AAA$ and $\YY$ are nonnegative. (Note that the column vector representing $\vv$ would likewise include $\YY$.)
 By Lemma \ref{negdefinverse}, all entries of 
${\MM}^{-1}$ are nonpositive. Thus all entries of $\XX$ are nonnegative.
\end{proof}
The following more technical lemma is akin to Lemma 14.12 of \cite{bad}, but we give a more elementary proof.
\begin{lemma} \label{thirdlem}
Suppose that $N \subset W$ are two special subspaces, with $N$ being negative definite. Suppose there is an effective vector $\vv\in V$ with the following properties:
\begin{enumerate}
\item $\vv \cdot \ee \leq 0$ for each basis element $\ee \in N$;
\item $\vv \cdot \ee < 0$ for each basis element in $\ee \in W \setminus N$.
\end{enumerate}
Then $W$ is also a negative definite subspace.
\end{lemma}
\begin{proof}
We give a proof by contradiction. Suppose that $W$ is not negative definite. Then by Corollary \ref{restate} there is a nonzero effective and nef element $\qq$ in $W$. Since $N$ is a negative definite subspace, $\qq \notin N$.
Thus $\vv \cdot \qq < 0$, but this contradicts the fact that $\qq$ is nef.
\end{proof}
Here is Zariski's algorithm for the decomposition of a specified effective vector $\vv$. 
If $\vv$ is nef, then the decomposition is given by $\pp=\vv$ and $\nn=\zzero$.
Otherwise let $N_1$ be the subspace spanned by all basis vectors $\ee$ for which 
$\vv \cdot \ee < 0$. Since $\vv$ is effective, $N_1$ is a subspace of its support space and hence has finite dimension.
By Lemma \ref{thirdlem} (with $W=N_1$ and $N$ trivial), it is a negative definite subspace. Since the restriction of the intersection product to $N_1$ is nonsingular, there is a unique vector $\nn_1 \in N_1$ satisfying this system of equations:
$$
\nn_1 \cdot \ee = \vv \cdot \ee \quad \text{for each basis vector $\ee \in N_1$}.
$$
By Lemma \ref{firstlem}, $\nn_1$ is effective. Let $\vv_1=\vv-\nn_1$, which by Lemma \ref{secondlem} is an effective vector.
If $\vv_1$ is nef with respect to $V$, then we have found the Zariski decomposition: 
$\pp=\vv_1$ and $\nn=\nn_1$.
\par
Otherwise proceed inductively as follows. By an inductive hypothesis, $\vv_{k-1}$ is an effective vector satisfying $\vv_{k-1} \cdot \ee=0$ for each basis vector $\ee\in N_{k-1}$. Let $N_k$ be the subspace spanned by $N_{k-1}$ and by all basis vectors $\ee$ for which $\vv_{k-1}\cdot\ee<0$. Again $N_k$ is finite-dimensional. By Lemma \ref{thirdlem} (with $N=N_{k-1}$ and $W=N_k$), the subspace $N_k$ is negative definite. Hence there is a unique vector $\nn_k \in N_k$ satisfying this system of equations:
$$
\nn_k \cdot \ee = \vv_{k-1} \cdot \ee \quad \text{for each basis vector $\ee \in N_k$}.
$$
By Lemma \ref{firstlem}, $\nn_k$ is effective. Let $\vv_k=\vv_{k-1}-\nn_k$, which is effective by Lemma \ref{secondlem}. If $\vv_k$ is nef with respect to $V$, then the Zariski decomposition is $\pp=\vv_k$ and
$\nn=\nn_1+ \cdots + \nn_k$. Otherwise $\vv_k \cdot \ee = 0$ for each basis vector 
$\ee\in N_k$, which is the required inductive hypothesis. Since the sequence of subspaces $N_1 \subset N_2 \subset \cdots$ is strictly increasing and contained in the support space of $\vv$, this process eventually terminates.
\par
\begin{example}\label{algexample}
Using the same intersection matrix as in Example \ref{algoex}, we apply Zariski's algorithm to the vector 
$$
\vv=
\left[
\begin{array}{c}
8 \\
4 \\
5 \\
9
\end{array}
\right].
$$
Here $N_1$ is spanned by $\ee_1$ and $\ee_4$, and 
$$
\nn_1=
\left[
\begin{array}{c}
2 \\
0 \\
0 \\
2
\end{array}
\right].
$$
Since the complementary vector
$\vv-\nn_1$ is nef, the Zariski decomposition is as follows:
$$
\left[\begin{array}{c}
8 \\ 4 \\ 5 \\ 9
\end{array}\right]
=
\left[\begin{array}{c}
6 \\ 4 \\ 5 \\ 7
\end{array}\right]
+
\left[\begin{array}{c}
2 \\ 0 \\ 0 \\ 2
\end{array}\right].
$$
Thus the algorithm works in just one step.
\par
For the vector 
$$
\vv=
\left[
\begin{array}{c}
4 \\
2 \\
3 \\
6
\end{array}
\right],
$$
however, the algorithm requires three steps, as follows: $N_1$ is spanned by the single vector $\ee_4$, and 
$$
\nn_1=
\left[
\begin{array}{c}
0 \\
0 \\
0 \\
2
\end{array}
\right].
$$
The vector $\vv_1=\vv-\nn_1$ is not nef. We find that $N_2$ is spanned by $\ee_1$ and $\ee_4$, and that 
$$
\nn_2=
\left[
\begin{array}{c}
2/3 \\
0 \\
0 \\
1/3
\end{array}
\right].
$$
Again $\vv_2=\vv_1-\nn_2$ is not nef. Now $N_3$ is spanned by $\ee_1$, $\ee_3$, and $\ee_4$, with 
$$
\nn_3=
\left[
\begin{array}{c}
1/3 \\
0 \\
1/2 \\
1/6
\end{array}
\right],
$$
so that $\vv_3=\vv_2-\nn_3$ is nef.
The Zariski decomposition is
$$
\left[\begin{array}{c}
4 \\ 2 \\ 3 \\ 6
\end{array}\right]
=
\left[\begin{array}{c}
3 \\ 2 \\ 5/2 \\ 7/2
\end{array}\right]
+
\left[\begin{array}{c}
1 \\ 0 \\ 1/2 \\ 5/2
\end{array}\right].
$$
\end{example}


\section{Numerical equivalence.}
We continue to suppose that $V$ is a vector space over $\QQ$ equipped with an intersection product with respect to a fixed basis. We say that two elements $\vv$ and $\vv'$ of $V$ are \emph{numerically equivalent in $V$} if 
$\vv \cdot \ww=\vv' \cdot \ww$ for each element $\ww \in V$. A vector numerically equivalent in $V$ to $\zzero$ is said to be \emph{numerically trivial in $V$}.
\begin{proposition} \label{numeqthm}
Suppose that $\vv$ and $\vv'$ are effective vectors which are numerically equivalent in $V$.
Let $\vv=\pp+\nn$ and $\vv'=\pp'+\nn'$ be their Zariski decompositions.
Then $\nn=\nn'$.
\end{proposition}
\begin{proof}
Note that $\vv'-\nn$ is numerically equivalent to $\pp$. Thus $\vv'=(\vv'-\nn)+\nn$ satisfies all four requirements for a Zariski decomposition of  $\vv'$.
By uniqueness of this decomposition, we must have 
$\nn=\nn'$.
\end{proof}
\begin{example} \label{neqexample}
Suppose that $V$ is a 5-dimensional vector space with intersection matrix
$$
\left[
\begin{array}{rrrrr}
-2 & 1 & 1 & 1 & 1\\
1 & -1 & 0 & 0 & 0\\
1 & 0 & -1 & 0 & 0\\
1 & 0 & 0 & -1 & 0\\
1 & 0 & 0 & 0 & 1
\end{array}
\right].
$$
Let $\vv=3\ee_1+\ee_2+\ee_3+\ee_4$, and let $\vv'$ be the numerically equivalent vector $2\ee_1+\ee_5$. Then the Zariski decompositions are as follows:
$$
\pp=\frac{3}{2}\ee_1+\ee_2+\ee_3+\ee_4,
\qquad
\pp'=\frac{1}{2}\ee_1+\ee_5,
\qquad
\nn=\nn'=\frac{3}{2}\ee_1.
$$
\end{example}
Using the notion of numerical equivalence, we can extend Zariski decomposition to a potentially larger set of vectors. We say that a vector $\ww \in V$ is \emph{quasi-effective in $V$} if $\ww \cdot \vv \geq 0$ for every element $\vv\in V$ which is nef with respect to $V$.\footnote{We have heard ``quef'' as a short form. The terminology ``pseudo-effective'' is also in use. }
In particular each effective element is quasi-effective; more generally, any vector numerically equivalent to an effective vector is quasi-effective.
\begin{proposition}
Suppose that $\MM$ is an intersection matrix for a finite-dimensional vector space 
$V$. Then $\ww$ is quasi-effective in $V$ if and only if $\ww^T \MM \vv \geq 0$ whenever $\MM\vv$ is effective. In particular if $\MM$ is nonsingular, then $\ww$ is quasi-effective in $V$ if and only if it is effective.
\end{proposition}
\begin{proof}
The first sentence uses the definitions, together with the previous observation that a vector $\vv$ is nef with respect to $V$ if and only if $\MM\vv$ is effective. If the matrix is nonsingular then each effective element can be written as $\MM\vv$ for some nef element $\vv$. Thus in this case $\ww$ is quasi-effective in $V$ if and only if $\ww^T \vv \geq 0$ for each effective element $\vv$. An element $\ww$ satisfying the latter condition must be effective.
\end{proof}
In general, however, there may be quasi-effective vectors which are not effective. 
In Example \ref{neqexample}, for instance, the vector $\ww=\frac{7}{2}\ee_1+\frac{3}{2}\ee_2+\frac{3}{2}\ee_3+\frac{3}{2}\ee_4-\frac{1}{2}\ee_5$ is quasi-effective, since it is numerically equivalent to the effective vector $2\ee_1+\ee_5$.
\par
Here is another example, which shows that the notion of quasi-effectiveness is ``volatile'' as one passes to subspaces.
\begin{example}
Start with the following $(2k)\times (2k)$ matrix.
\[
\PP_{2k} = \left[
\begin{array}{ccccccc}
1 & 0 & 1 & 0 & \cdots & 1 & 0 \\
0 & 0 & 1 & 0 & \cdots & 1 & 0  \\
1 & 1 & 1 & 0 & \cdots & 1 & 0  \\
0 & 0 & 0 & 0 & \cdots & 1 & 0  \\
\vdots & \vdots & \vdots & \vdots & \ddots & \vdots & \vdots  \\
1 & 1 & 1 & 1 & \cdots & 1 & 0  \\
0 & 0 & 0 & 0 & \cdots & 0 & 0
\end{array}
\right]
\]
Use row and column operations to construct an intersection matrix
$\MM_{2k}$ as follows: beginning at $i=2$, replace column $2i$ by itself plus column 1 minus column 2, and do the corresponding operation on row $2i$;
continue this up until $i=k$.
Here is an illustration when $k=3$:
\[
\PP_6 = \left[
\begin{array}{cccccc}
1 & 0 & 1 & 0 & 1 & 0 \\
0 & 0 & 1 & 0 & 1 & 0 \\
1 & 1 & 1 & 0 & 1 & 0 \\
0 & 0 & 0 & 0 & 1 & 0 \\
1 & 1 & 1 & 1 & 1 & 0 \\
0 & 0 & 0 & 0 & 0 & 0
\end{array}
\right]
\quad\text{and}\quad \MM_6 = \left[
\begin{array}{cccccc}
1 & 0 & 1 & 1 & 1 & 1 \\
0 & 0 & 1 & 0 & 1 & 0 \\
1 & 1 & 1 & 0 & 1 & 0 \\
1 & 0 & 0 & 1 & 1 & 1 \\
1 & 1 & 1 & 1 & 1 & 0 \\
1 & 0 & 0 & 1 & 0 & 1
\end{array}
\right].
\]
Let $\MM_j$ and $\PP_j$ denote the upper left  $j\times j$ submatrices
(noting that this is consistent with our previous usage when $j$ is even).
Note that $\det \MM_j=\det \PP_j$ for all $j$. In particular the
matrix $\MM_{j}$ is singular if and only if $j$ is even.
\par
Now let $V_j$
denote the subspace spanned by the first $j$ basis vectors, and
consider the vector $\ww=\ee_1-\ee_2$. If $j>1$ is odd, then 
$\ww$ is not quasi-effective in $V_j$, since the
matrix is nonsingular and $\ww$ is not effective.
If $j>2$ is even, however, then $\ww$ is numerically equivalent in $V_j$ to the effective vector
$\ee_j$; hence $\ww$ is quasi-effective in $V_j$.
(It's also quasi-effective in $V_2$, being numerically equivalent to $\ee_1$.)
\end{example}
\begin{proposition}
If a vector $\ww$ is numerically equivalent to an effective vector, then it has a unique Zariski decomposition, i.e., there is unique way to write it as a sum of a nef vector $\pp$ and an effective vector $\nn$ satisfying conditions (1) through (4) of Theorem \ref{zdecomp}.
\end{proposition}
Note, however, that the positive part does not have to be effective. In particular if $\ww$ is nef but not effective, then its positive part is itself.
\begin{proof}
Suppose that $\ww=\vv+\ttt$, where $\vv$ is effective and $\ttt$ is numerically trivial, and let 
$\vv=\qq+\nn$ be the Zariski decomposition of $\vv$. Putting $\pp=\qq+\ttt$, we see that $\pp$ and $\nn$ satisfy the four conditions. Conversely, if $\ww=\pp+\nn$ is a Zariski decomposition then $\vv=(\pp-\ttt)+\nn$ must be the unique decomposition of $\vv$. Thus the Zariski decomposition of $\ww$ is unique.
\end{proof}
For a detailed treatment of Zariski decomposition for quasi-effective vectors (in the original context, where these vectors represent curves on surfaces), see~\cite{fujita}.



\section{The original context (continued).} \label{windup}
We now resume our informal account of the original context in which Zariski developed his theory of decomposition. Figure \ref{twocubics} shows two plane curves of degree three. The polynomial
\begin{equation} \label{reducible}
f(x,y)=(y-x^2)(3y-x-3)
\end{equation} 
defining the curve on the right can be factored, with the visible result that the curve is the union of a line and a conic (a curve of degree 2); we say that these are the \emph{components} of the curve. The other curve has a single component: we call it \emph{irreducible}.
\par
Suppose that
\begin{equation*}
f(x,y)=(f_1(x,y))^{n_1}(f_2(x,y))^{n_2}\cdots(f_k(x,y))^{n_k},
\end{equation*} 
where each $f_i$ is an irreducible polynomial and thus defines an irreducible curve $C_i$, one of the components of $f$. We associate to $f$ the formal linear combination
$$
D=\sum_{i=1}^k n_i C_i,
$$
calling it the \emph{divisor} of $f$. Note that all coefficients are nonnegative; thus this is an \emph{effective divisor}. For example, the divisor associated to the polynomial in (\ref{reducible}) is $C_1+C_2$, where $C_1$ is the conic and $C_2$ is the line.
A similar recipe works for any other surface. For an effective divisor in the plane we define its \emph{degree} to be the degree of the defining polynomial; thus the degree of $\sum n_i C_i$ is $\sum n_i \deg{C_i}$.
\par
Given two distinct irreducible curves $C$ and $D$ on an algebraic surface, they have an \emph{intersection number} $C \cdot D$. Intuitively, this is the number of points in which the curves intersect, and indeed in many cases that is its precise meaning, but to define this number carefully one needs to consider exotic possibilities, so that for example a tangency between the curves gets counted as ``two points" (or even more). Thus to an algebraic curve we can associate a matrix recording the intersection numbers of its components. In the plane\footnote{As in Section \ref{origcontext}, we mean the complex projective plane.}
the intersection number between curves of degrees $c$ and $d$ is $cd$, a fundamental result of \'Etienne B\'ezout dating to 1776.
Hence for the curves in Figure \ref{twocubics} these matrices are
$$
\left[
\begin{array}{c}
9 \\
\end{array}
\right]
\qquad
\text{and}
\qquad
\left[
\begin{array}{cc}
4 & 2 \\
2 & 1 \\
\end{array}
\right].
$$
\par
\begin{figure}
\begin{center}
      \scalebox{0.40}
      {\includegraphics{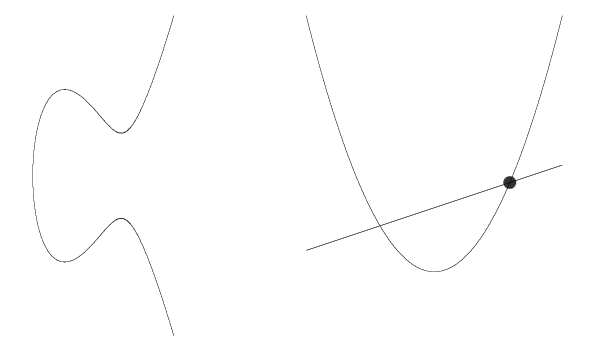}}     
\caption{Two plane curves of degree three. The curve on the left is irreducible, while the curve on the right has two components. In the next figure, we show what happens if this curve is blown up at the indicated point.}
\label{twocubics}
\end{center}
    \end{figure}
\par
For distinct irreducible curves, the intersection number is always a nonnegative integer. Thus the off-diagonal entries in these matrices are nonnegative, and they are intersection matrices as defined in Section  \ref{decomp}. The diagonal entries are \emph{self-intersection numbers}. In our example we have calculated them using B\'ezout's formula, but on other algebraic surfaces one has the startling fact: a self-intersection number may be negative! The simplest example of this comes from a process called \emph{blowing up}, in which a given surface is modified by replacing one of its points $p$ by an entire curve $E$ having self-intersection number $-1$, called an \emph{exceptional curve}. (This process is the basic operation of the ``birational geometry'' to which Mumford alludes in the quotation in Section \ref{origcontext}.) Each irreducible curve $C$ on the original surface which contains $p$ can be ``lifted'' to a curve on the new surface meeting $E$. We will abuse notation by referring to the lifted curve with the same letter $C$, but a remarkable thing happens to $C \cdot C$: it is reduced in value (typically by 1). For example, if one blows up the plane at one of the two intersection points shown in Figure \ref{twocubics}, then the intersection matrix for the two original components and the new curve $E$ is as follows:
$$
\left[
\begin{array}{rrr}
3 & 1 & 1 \\
1 & 0 & 1 \\
1 & 1 & -1 \\
\end{array}
\right].
$$
See Figure \ref{blownup}, noting that the two original components have been pulled apart, so that they now meet at just a single point. 
\par

\begin{figure}
\begin{center}
      \scalebox{0.40}
      {\includegraphics{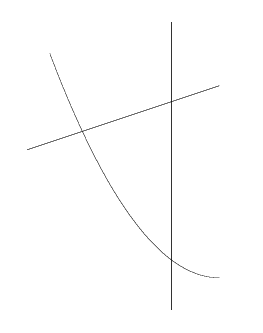}}     
\caption{The result of blowing up the rightmost curve in Figure \ref{twocubics} at the indicated point. The exceptional curve is represented by a vertical line.}
\label{blownup}
\end{center}
    \end{figure}
\par
The definition of intersection number between a pair of irreducible curves extends by linearity to any pair of divisors (effective or not). If one has the intersection matrix, then the calculation is simply a matrix multiplication. For example, the self-intersection of the divisor $C_1+C_2$
associated to the polynomial in (\ref{reducible}) is
$$
\left[
\begin{array}{cc}
1 &1 \\
\end{array}
\right]
\left[
\begin{array}{cc}
4 & 2 \\
2 & 1 \\
\end{array}
\right]
=
\left[
\begin{array}{c}
1 \\
1 \\
\end{array}
\right]
=9.
$$
Note that the result is the square of its degree.
\par
As we have observed, the matrix of intersection numbers for a set of irreducible curves on an algebraic surface is an intersection matrix. Thus for any effective divisor $D$ we can compute a Zariski decomposition, obtaining a positive and negative part.
What Zariski discovered in his fundamental paper \cite{zar} is that the solution of the Riemann-Roch problem for $D$ was strongly controlled by its positive part. More precisely, letting $P$ denote the positive part, he showed that
\begin{equation}\label{zf}
\lim_{n \to \infty}\frac{\dim |nD|}{n^2/2}=P \cdot P.
\end{equation}
To illustrate this formula, we present two examples.
\begin{example}
Let $D$ be an effective divisor of degree $d$ in the plane. The linear system $|nD|$ consists of all effective divisors of degree $nd$, and thus has dimension $\binom{nd+2}{2}-1$. By B\'ezout's theorem, the intersection of $D$ with any irreducible curve is positive; hence $D$ is nef, and thus its positive part is $D$ itself. Zariski's formula (\ref{zf}) says that 
$$
\lim_{n \to \infty}\frac{\binom{nd+2}{2}-1}{n^2/2}=D \cdot D=d^2.
$$
\end{example}
\begin{example}
(This example is also treated in Example~3.5 of \cite{chambers2}.)
Blow up the plane at two points $P_1$ and $P_2$, calling the exceptional curves $E_1$ and $E_2$, and let $L$ denote the lift of the line through the two points. Then the intersection matrix with respect to the ordered basis 
$\{L,E_1,E_2\}$
 is
$$
\left[
\begin{array}{rrr}
-1 & 1 & 1 \\
1 & -1 & 0 \\
1 & 0 & -1 \\
\end{array}
\right].
$$
Consider $D=aL+bE_1+cE_2$, where all coefficients are nonnegative. Then there are five possibilities for the Zariski decomposition:
\begin{equation}\label{5cases}
\begin{cases}
(aL+bE_1+cE_2)+0 & \text{if} \quad a\geq b, \hspace{1mm} a\geq c,  \hspace{1mm} b+c \geq a\\
(aL+aE_1+aE_2)+((b-a)E_1+(c-a)E_2) & \text{if} \quad a\leq b, \hspace{1mm} a\leq c \\
(aL+aE_1+cE_2)+(b-a)E_1 & \text{if} \quad c\leq a \leq b \\
(aL+bE_1+aE_2)+(c-a)E_2 & \text{if} \quad b\leq a \leq c\\
((b+c)L+bE_1+cE_2)+(a-(b+c))L & \text{if} \quad b+c\leq a
\end{cases}
\end{equation}
(where we have always written the positive part first).
\par
We can give a description of the linear system $|D|$ in terms of plane curves, as follows: it consists of those curves $f(x,y)=0$ for which the polynomial $f$ has degree $a$ and satisfies these conditions:
\begin{enumerate}
\item the partial derivatives of $f$ of order less than $a-b$ vanish at $P_1$;
\item similarly, the partial derivatives of $f$ of order less than $a-c$ vanish at $P_2$.
\end{enumerate}
Let us check this description against Zariski's formula  (\ref{zf}) in the first and last of the five cases of (\ref{5cases}) (the other cases being similar). In the first case, we are imposing $\binom{a-b+1}{2}$ conditions at the point $P_1$, and $\binom{a-c+1}{2}$ conditions at $P_2$. One can confirm that the two sets of conditions are independent, and thus the dimension of the linear system is 
$$
\dim |D|=\binom{a+2}{2}-\binom{a-b+1}{2}-\binom{a-c+1}{2}-1.
$$
Similarly one has
$$
\dim |nD|=\binom{na+2}{2}-\binom{n(a-b)+1}{2}-\binom{n(a-c)+1}{2}-1,
$$
so that 
$$
\lim_{n \to \infty}\frac{\dim |nD|}{n^2/2}=a^2-(a-b)^2-(a-c)^2=-a^2-b^2-c^2+2ab+2ac=P \cdot P.
$$
\par
In the final case of (\ref{5cases}), the conditions imposed at the two points are no longer independent. However one can show the following: each polynomial $f$ is divisible by $l^{a-(b+c)}$, where $l=0$ is an equation of the line through $P_1$ and $P_2$; furthermore, the quotient $f/l^{a-(b+c)}$ has degree $b+c$, with its partial derivatives of order less than $b$ vanishing at  $P_1$, and similarly its partial derivatives of order less than $c$ vanishing at  $P_2$; these conditions are independent, and thus the dimension of the linear system is
$$
\dim |D|=
\binom{b+c+2}{2}-\binom{b+1}{2}-\binom{c+1}{2}-1
=(b+1)(c+1)-1.
$$
Similarly
$$
\dim |nD|=(nb+1)(nc+1)-1,
$$
so that 
$$
\lim_{n \to \infty}\frac{\dim |nD|}{n^2/2}=2bc=-(b+c)^2-b^2-c^2+2(b+c)b+2(b+c)c=P \cdot P.
$$
\par
\end{example}
Zariski's ideas about decomposition of curves on an algebraic surface continue to resonate in contemporary developments. Miles Reid \cite{reid}, for example, has written that ``Zariski's paper on the asymptotic form of Riemann-Roch for a divisor on a surface forms a crucial bridge between the Italian tradition of surfaces and modern work on 3-folds [algebraic varieties of dimension 3].''
It led Reid, Mori, Koll\'ar, and other researchers to the crucial ideas of ``extremal rays'' and ``canonical and minimal models'' in higher dimensions. Reid emphasizes that ``the Zariski decomposition of a divisor on a surface is \dots a kind of minimal model program.'' For an introduction to these modern aspects of higher-dimensional algebraic geometry, see \cite{matsuki}.

\paragraph{Acknowledgments.}
We learned a great deal about this topic through conversations with Herb Clemens and from the text of Robert Lazarsfeld \cite{laz}. We also thank Lazarsfeld for advice on how to rearrange this paper.


 
 \bigskip

\noindent\textbf{Thomas Bauer}  is Professor of Mathematics at Philipps Universit\"at
Marburg. He received his Ph.D. and his habilitation from the
University of Erlangen-N\"urnberg. His primary research lies in
Algebraic Geometry. Moreover, he has a strong interest in the
education of math teachers.

\noindent\textit{Fachbereich Mathematik und Informatik, Philipps-Universit\"at Marburg, Hans-Meerwein-Stra{\ss}e, Lahnberge, D-35032 Marburg, Germany\\
tbauer@mathematik.uni-marburg.de}

\bigskip

\noindent\textbf{Mirel Caib\u{a}r}  received his Ph.D. from the  University of Warwick in $1999$. His research area is Algebraic Geometry. He is currently an Assistant Professor at the Mansfield campus of The Ohio State University. 

\noindent\textit{Ohio State University at Mansfield, 1760 University Drive,
Mansfield, Ohio 44906, USA\\
caibar@math.ohio-state.edu}

\bigskip

\noindent\textbf{Gary Kennedy}  is Professor of Mathematics at the Mansfield campus of The Ohio State University. He received his Ph.D. from Columbia University in 1981. Together with his son Stephen, he has twice constructed a daily crossword puzzle for the New York Times.

\noindent\textit{Ohio State University at Mansfield, 1760 University Drive,
Mansfield, Ohio 44906, USA\\
kennedy@math.ohio-state.edu}


\begin{thebibliography}{99}

\bibitem{bad} L.~B\u{a}descu,  {\em Algebraic Surfaces},  Springer-Verlag, New York, 2001.

\bibitem{simpleproof} T. Bauer, A simple proof for the existence of Zariski decompositions on surfaces,  {\em J. Algebraic Geom.} {\bf 18} (2009) 789--793.

\bibitem{chambers} T.~Bauer and M.~Funke, Weyl and Zariski chambers on K3 surfaces, {\em Forum Math.} (to appear).

\bibitem{chambers2} T.~Bauer,  A.~K\"{u}ronya, and T.~Szemberg, Zariski chambers, volumes, and stable base loci,  {\em J. Reine Angew. Math.} {\bf 576} (2004) 209--233. 

\bibitem{fujita}  T.~Fujita, On Zariski problem, {\em Proc. Japan Acad. Ser. A Math. Sci.} {\bf 55} (1979) 106--110.

\bibitem{laz} R.~Lazarsfeld,  {\em Positivity in Algebraic Geometry}, vols. 1 and 2,  Springer-Verlag, Berlin, 2004. 

\bibitem{gen} Mathematics Genealogy Project, North Dakota State University Department of Mathematics, Fargo, available at http://genealogy.math.ndsu.nodak.edu.

\bibitem{matsuki}   K.~Matsuki,  {\em Introduction to the Mori Program}, Universitext, Springer, New York, 2002. 

\bibitem{moriwaki} A.~Moriwaki, Zariski decompositions on arithmetic surfaces,    arXiv:0911.2951v1, accessed on Nov. 18, 2009.

\bibitem{parikh} C.~Parikh,  {\em The Unreal Life of Oscar Zariski}, Academic Press, Boston, 1991.

\bibitem{reid} M.~Reid, Twenty five years of $3$-folds---an old person's view,  in {\em Explicit Birational Geometry of 3-Folds}, London Math. Soc. Lecture Note Ser., vol. 281, Cambridge University Press, Cambridge, 2000, 313--343.

\bibitem{zar} O.~Zariski,  The theorem of Riemann-Roch for high multiples of an effective divisor on an algebraic surface,  {\em Ann. of Math.} {\bf 76} (1962) 560--615.

\bibitem{algsurf}  \textemdash,  {\em Algebraic Surfaces}, 2nd ed.,  Springer-Verlag, New~York, 1971.





\end{thebibliography}
\end{document}